\documentclass[11pt,reqno]{amsart}
\usepackage{amssymb,latexsym}
\usepackage{graphicx}
\usepackage{fancyhdr,amssymb}
\usepackage{url}		

\newcommand{\beq}{\begin{equation}}
\newcommand{\eeq}{\end{equation}}
\newcommand{\interval}{[K_n, K_{n + 1})}

\newcommand{\Xn}{X_{i, i +j}(n)}

\newcommand{\Ln}{L_{i, i +j}(n)}
\newcommand{\Rn}{R_{i, i +j}(n)}

\newcommand{\glgl}{\lambda_{g,\ell}}

\newcommand\be{\begin{equation}}
\newcommand\ee{\end{equation}}
\newcommand\bea{\begin{eqnarray}}
\newcommand\eea{\end{eqnarray}}

\newcommand\bel{\begin{equation}}
\newcommand\eel{\end{equation}}
\newcommand\beal{\begin{eqnarray}}
\newcommand\eeal{\end{eqnarray}}

\newcommand\bi{\begin{itemize}}
\newcommand\ei{\end{itemize}}
\newcommand\ben{\begin{enumerate}}
\newcommand\een{\end{enumerate}}

\newcommand{\twocase}[5]{#1 \begin{cases} #2 & \text{{\rm #3}}\\ #4 &\text{{\rm #5}} \end{cases}   }
\newcommand{\ncr}[2]{{#1 \choose #2}}
\newcommand{\threecase}[7]{#1 \begin{cases} #2 &
\text{\rm #3}\\ #4 &\text{\rm #5}\\ #6 &\text{\rm #7} \end{cases}   }

\newtheorem{thm}{Theorem}[section]

\newtheorem{lem}[thm]{Lemma}
\newtheorem{rek}[thm]{Remark}

\newtheorem{defi}[thm]{Definition}

\numberwithin{equation}{section}

\newtheorem*{remark*}{Remark}
\numberwithin{remark}{section}
\numberwithin{subsubcase}{subcase}
\numberwithin{subsubsection}{subsection}

\begin{document}
\fancyhead{}
\renewcommand{\headrulewidth}{0pt}
\fancyfoot{}
\fancyfoot[LE,RO]{\medskip \thepage}
\fancyfoot[LO]{\medskip MONTH YEAR}
\fancyfoot[RE]{\medskip VOLUME , NUMBER }

\setcounter{page}{1}

\title[The Average Gap Distribution for Generalized Zeckendorf Decompositions]{The Average Gap Distribution for Generalized Zeckendorf Decompositions}
\author{Olivia Beckwith}
\address{Department of Mathematics, Harvey Mudd College, Claremont, CA 91711}\email{obeckwith@gmail.com}
\thanks{The third and sixth named authors were partially supported by NSF grant DMS0970067 and the remaining authors were partially supported by NSF Grant DMS0850577. It is a pleasure to thank our colleagues from the Williams College 2010, 2011 and 2012 SMALL REU program for many helpful conversations, Florian Luca and the referee for comments on earlier drafts, and Philippe Demontigny for discussions on generalizations.}
\author{Amanda Bower}
\address{Department of Mathematics
and Statistics, University of Michigan-Dearborn, Dearborn, MI 48128}\email{amandarg@umd.umich.edu}

\author{Louis Gaudet}\address{Department of Mathematics, Yale University, New Haven, CT 06510}\email{louis.gaudet@yale.edu} 

\author{Rachel Insoft}
\address{Department of Mathematics, Wellesley College,  Wellesley, MA 02481}\email{rinsoft@wellesley.edu}

\author{Shiyu Li}
\address{Department of Mathematics, University of California, Berkeley, Berkeley, CA 94720}\email{jjl2357@berkeley.edu}

\author{Steven J. Miller}
\address{Department of Mathematics and Statistics, Williams College, Williamstown, MA 01267}
\email{sjm1@williams.edu, Steven.Miller.MC.96@aya.yale.edu}

\author{Philip Tosteson}
\address{Department of Mathematics and Statistics, Williams College, Williamstown, MA 01267}\email{Philip.D.Tosteson@williams.edu}

\begin{abstract}
An interesting characterization of the Fibonacci numbers is that, if we write them as $F_1 = 1$, $F_2 = 2$, $F_3 = 3$, $F_4 = 5, \dots$, then every positive integer can be written uniquely as a sum of non-adjacent Fibonacci numbers. This is now known as Zeckendorf's theorem \cite{Ze}, and similar decompositions exist for many other sequences $\{G_{n+1} = c_1 G_{n} + \cdots + c_L G_{n+1-L}\}$ arising from recurrence relations. Much more is known. Using continued fraction approaches, Lekkerkerker \cite{Lek} proved the average number of summands needed for integers in $[G_n, G_{n+1})$ is on the order of $C_{{\rm Lek}} n$ for a non-zero constant; this was improved by others to show the number of summands has Gaussian fluctuations about this mean.

Kolo$\breve{{\rm g}}$lu, Kopp, Miller and Wang \cite{KKMW,MW1} recently recast the problem combinatorially, reproving and generalizing these results. We use this new perspective to investigate the distribution of gaps between summands. We explore the average behavior over all $m \in [G_n, G_{n+1})$ for special choices of the $c_i$'s. Specifically, we study the case where each $c_i \in \{0,1\}$ and there is a $g$ such that there are always exactly $g-1$ zeros between two non-zero $c_i$'s; note this includes the Fibonacci, Tribonacci and many other important special cases. We prove there are no gaps of length less than $g$, and the probability of a gap of length $j > g$ decays geometrically, with the decay ratio equal to the largest root of the recurrence relation. These methods are combinatorial and apply to related problems; we end with a discussion of similar results for far-difference (i.e., signed) decompositions.
\end{abstract}

\maketitle

\tableofcontents

\bigskip
\bigskip

\section{Introduction}

We begin by reviewing earlier, related results, and then state the special sequences to be studied in this paper and our results. Though we study  special cases in order to simplify some of the proofs and prevent the ideas from being obscured by technical calculations, we describe the problems in as much generality as possible in the introduction for the benefit of the reader interested in generalizations.

A beautiful theorem of Zeckendorf \cite{Ze} states that every positive integer can be written uniquely as a sum of non-adjacent Fibonacci numbers if we label them $F_1 = 1$, $F_2 = 2$, $F_3 = 3$, $F_4 = 5$ and so on; we must use this non-standard ordering as otherwise the decomposition would not be unique. The standard proof of this is through a greedy algorithm, though this does not provide any information about finer questions. Using continued fraction techniques, Lekkerkerker \cite{Lek} proved that the average number of summands needed for decompositions of numbers in $[F_n, F_{n+1})$ is $n/(\phi^2+1) + O(1)$, where $\phi = (1+\sqrt{5})/2$ is the golden mean. These results have been extended to other sequences, specifically $\{G_n\}$ being a positive linear recurrence sequence. This means that there exist non-negative integers $c_i$ (with $c_1, c_L$ positive) such that \be G_{n+1} \ = \ c_1 G_{n} + \cdots + c_L G_{n+1-L}, \ee as well as rules to specify the first $L$ terms of the sequence and a generalization of the non-adjacency constraint to what is a `legal' decomposition (other authors use the phrase $G$-ary decomposition for a legal decomposition). Informally, a legal decomposition is one where we cannot use the recurrence relation to replace a linear combination of summands with another summand, and the coefficient of each summand is appropriately bounded. The precise result is as follows (see \cite{MW2} for example for a proof).


\begin{thm}[Generalized Zeckendorf Decomposition]\label{lem:legaldecomp} Consider a positive linear recurrence \be G_{n+1} \ = \ c_1 G_{n} + \cdots + c_L G_{n+1-L} \ee with non-zero integer coefficients, $c_1, c_L > 0$, and initial conditions $G_1 = 1$ and for $1 \le n \le L$ \be G_{n+1} \ =\ c_1 G_n + c_2 G_{n-1} + \cdots + c_n G_{1}+1. \ee  For each positive integer $N$ there exists a unique decomposition $\sum_{i=1}^{m} {a_i G_{m+1-i}}$\label{a_i} with $a_1>0$, the other $a_i \ge 0$, and one of the following two conditions holds:

\begin{itemize}

\item We have $m < L$ and $a_i=c_i$ for $1\le i\le m$.

\item There exists $s\in\{1,\dots, L\}$ such that
\be\label{eq:legalcondition2}
a_1\ = \ c_1,\ a_2\ = \ c_2,\ \dots,\ a_{s-1}\ = \ c_{s-1}\ {\rm{and}}\ a_s<c_s,
\ee
$a_{s+1}, \dots, a_{s+\ell} \ = \  0$ for some $\ell \ge 0$,
and $\{b_i\}_{i=1}^{m-s-\ell}$ (with $b_i = a_{s+\ell+i}$) is either legal or empty.
\end{itemize}
\end{thm}

In addition to the generalization of Zeckendorf's Theorem, Lekkerkerker's theorem generalizes, and the average number of summands is $C_{{\rm Lek}} n + d$ for some $C_{{\rm Lek}} > 0$. Further, the fluctuations of the number of summands above the mean converges to Gaussian behavior; see \cite{BCCSW,DG,FGNPT,Ho,Ke,LT,Len,MW1,MW2,Ste1,Ste2}.

In this paper we pursue a related question, the distribution of gaps between summands in generalized Zeckendorf decompositions. We adopt the perspective of Kolo$\breve{{\rm g}}$lu, Kopp, Miller and Wang \cite{KKMW,MW1,MW2}, which we briefly review. They proved Gaussian behavior by recasting the problem combinatorially, writing down explicit formulas for the number of $m \in [G_n, G_{n+1})$ with exactly $k$ summands through the cookie (or stars and bars) problem. Specifically, in the Fibonacci case consider all integers $m \in [F_n, F_{n+1})$ with exactly $k+1$ summands. If these summands are $F_{i_1}, F_{i_2}, \dots, F_{i_{k+1}=n}$ then the non-adjacency condition implies the gaps are $d_j := i_{j} - i_{j-1} \ge 2$ for $j \ge 2$ and $d_1 := i_1 \ge 1$. Thus $d_1 + \cdots + d_{k+1} = n$. As $d_1 \ge 1$ and $d_j \ge 2$ for $j \ge 2$, this is equivalent to solving $x_1 + \cdots + x_{k+1} = n - (2k+1)$, with each $x_i$ a non-negative integer. By the cookie (or stars and bars) problem, there are $\ncr{n-(2k+1) + k+1 - 1}{k+1-1}$ $=$ $\ncr{n-k-1}{k}$ integral solutions, and hence the number of integers in $[F_n, F_{n+1})$ with exactly $k+1$ summands is $\ncr{n-k-1}{k}$. As there are $F_{n+1}-F_n = F_{n-1}$ integers in $[F_n, F_{n+1})$, the probability density function for having $k+1$ summands is $\ncr{n-k-1}{k}/F_{n-1}$. In this case the proof is completed by using Stirling's formula to see convergence to a Gaussian. The general case is harder, and proceeds via generating functions.

Returning to the gaps, our goal is to analyze the distribution of the $d_j$'s. Given a decomposition \be m \ = \ G_{i_k} + G_{i_{k-1}} + \cdots + G_{i_2} + G_{i_1}, \ee we define the gaps to be \be i_k - i_{k-1}, \ \ \ i_{k-1} - i_{k-2}, \ \ \ \dots, \ \ \ i_2 - i_1; \ee note we do not consider $i_1 - 0$ as a gap (as there are on the order of $n (G_{n+1} - G_n)$ summands, one additional gap for each $m$ has a negligible affect). We look at the behavior of the average gap measure for special recurrences. By this we mean we amalgamate all gaps from all decompositions of all integers in $[G_n, G_{n+1})$, and show that as $n\to\infty$ this average converges to a limiting distribution which decays geometrically, with decay ratio equal to the largest root of the recurrence of the $G_n$'s. We concentrate on the class of recurrence relations below for two reasons. In addition to simplifying the proofs and highlighting the method, more can be proved about the limiting behavior in these cases then in general, and it is thus worth isolating these results. We consider what we shall name Kangaroo recurrences.


\begin{defi} Fix integers $\ell, g \ge 1$ and set $L = \ell g+1$. A \textbf{Kangaroo recurrence} is a positive linear recurrence relation where all the $c_i$'s are zero except $c_1 = c_{g+1} = c_{2g+1} = \cdots = c_L = 1$; thus \be K_{n+1} \ = \  K_n + K_{n-g} + K_{n-2g} + \cdots + K_{n-\ell g},\ee with initial conditions $K_1 = 1$, and for $1 \le n \le \ell g+1$
\be K_{n+1} \ =\
c_1 K_n + c_2 K_{n-1} + \cdots + c_n K_{1}+1. \ee Important special cases include the Fibonacci numbers ($g=1$ and $\ell = 1$), the Tribonaccis ($g=1$ and $\ell = 2$) and other similar recurrences. We call $g$ the \textbf{hop} length, and $\ell$ the number of hops. \end{defi}


For example, for the Tribonacci numbers the first terms are 1, 2, 4, 7, 13, while for $K_{n+1} = K_n + K_{n-2}$ the sequence begins 1, 2, 3, 4, 6, 9 (notice in this case we decompose 10 as $9+1$ and not $6+3+1$, as we can use the recurrence relation to replace $6+3$ with 9).\footnote{An alternative definition of the Fibonacci numbers are that they are the unique sequence such that every integer can be written uniquely as a sum of non-adjacent elements of the sequence. Similar equivalent definitions hold for the other relations studied.}

There is a nice, closed form expression for the general term in a Kangaroo recurrence. This expansion plays a central role in our investigations.
\begin{lem}[Generalized Binet's Formula]\label{lem:genbinet} Let $\lambda_1, \dots, \lambda_L$ be the roots of the characteristic polynomial of a Kangaroo recurrence with $\ell$ hops of length $g$ (and thus $L = \ell g$). Then $\lambda_1 > |\lambda_2| \ge \cdots \ge |\lambda_L|$, $\lambda_1 > 1$, and there exist constants such that \be K_n \ = \ a_1 \lambda_1^n + O\left(n^{L-2} \lambda_2^n\right). \ee More precisely, if $\lambda_1, \omega_2, \dots, \omega_r$ denote the distinct roots of the characteristic polynomial with multiplicities 1, $m_2, \dots, m_r$, then there are constants $a_1 > 0, a_{i,j}$ such that \be K_n \ = \ a_1 \lambda_1^n + \sum_{i=2}^r \sum_{j=1}^{m_r} a_{i,j} n^{j-1} \omega_i^n. \ee  \end{lem}

The main difficulty in the proof is showing that there is a unique root of largest absolute value, and that this root is positive. This follows from the Perron-Frobenius Theorem for irreducible matrices and some additional algebra; for completeness we provide an elementary proof in Appendix \ref{sec:proofgenbinet}, and then give the values of the roots and $a_1$ for three important Kangaroo recurrences (the Fibonaccis, the Tribonaccis, and the Skiponaccis).

The important take-aways from the above lemma are the following: an expansion exists, which is dominated by the unique, simple positive root of the characteristic polynomial (which is the largest root). We often write $\lambda_{g,\ell}$ for the largest root to emphasize its dependence on the parameters $g$ and $\ell$, which we vary in \S\ref{sec:approxglgl}.\\

To state our main result, we first need to set notation.

\begin{defi}[Gap indicator random variables, average gaps] Consider a Kangaroo recurrence $\{K_n\}$.

\begin{itemize}

\item Let $G_{i,i+j}(m;n)$ equal 1 if $K_i$ and $K_{i+j}$ are summands in the generalized Zeckendorf decomposition of $m \in \interval$ and $K_{i+r}$ is not a summand for $0 < r < j$, and $G_{i,i+j}(m;n)$ is zero otherwise. Thus $G_{i,i+j}(m;n)$ is 1 if and only if $m \in \interval$ has a gap of length \emph{exactly} $j$ starting at $K_i$.

\item Set \be X_{i,i+j}(n) \ := \ \sum_{m \in \interval} G_{i,i+j}(m;n); \ee $X_{i,i+j}(n)$ is the number of integers $m \in \interval$ that have a gap of length \emph{exactly}  $j$ starting at $K_i$. Let $P_n(j)$ be the proportion of all gaps from decompositions of $m \in \interval$ that are of length $j$.

\item Let $Y(n)$ denote the total number of gaps between summands of the generalized Zeckendorf decomposition of $m\in\interval$; thus \be Y(n) \ := \ \sum_{i=1}^n \sum_{j=0}^n X_{i,i+j}(n); \ee by the Generalized Lekkerkerker Theorem $Y(n)$ is $\left(C_{{\rm Lek}} n + O(1)\right)\left(K_{n+1}-K_n\right)$ (see \cite{MW1}).

\item The probability of a gap of length $j$, $P_n(j)$, is \be P_n(j) \ := \ \frac{\sum_{i=1}^{n-j} X_{i,i+j}(n)}{Y(n)}. \ee Finally set \be P(j)\ :=\ \lim_{n \to \infty} P_n(j),\ee so $P(j)$ represents the limiting probability that a gap in a decomposition for an integer in $\interval$ has length $j$.

\end{itemize}
\end{defi}

\noindent Our main result is the following.

\begin{thm}\label{thm:kangdistribution} Let $\{K_n\}$ be a Kangaroo recurrence of $\ell$ hops of length $g$, with $\glgl$ the largest root of the characteristic polynomial, $a_1$ the generalized Binet constant in front of $\glgl$, and $C_{\rm Lek}$ the computable Lekkerkerker constant. Then
\begin{equation}
\threecase{P(j)\ =\ }{0}{for $j < g$}{\left(\frac{a_1}{C_{\rm Lek}}\right) \lambda_{g,\ell}^{-2g} \frac{\glgl^{-g} - \glgl^{-\ell g}}{1 - \glgl^{-g}}}{for $j=g$}{(\lambda_{g,\ell}-1)^2 \left(\frac{a_1}{C_{\rm Lek}}\right) \lambda_{g,\ell}^{-j}}{for $j > g$,}
\end{equation} which implies the probability of a gap of length $j$ decays geometrically for $j > g$, with decay ratio equal to $\glgl$.
\end{thm}

\begin{rek}
As a quick consistency check, note that if $\ell = 1$ and $g=1$ then the probability of a gap of length 1 is zero; this is of course nothing more than a restatement of the definition of a legal decomposition by the shifted Fibonacci (Zeckendorf's theorem). More is true; if $\ell = 1$ there are never gaps of length $g$. This is good, as if there were a gap of length $g$ in this case we could immediately use the recurrence relation to replace it.
\end{rek}

\begin{rek} As the sum of the various gap probabilities must equal 1, we can write $a_1/C_{{\rm Lek}}$ in terms of the largest root $\lambda_{g,\ell}$. This allows us to compute $C_{{\rm Lek}}$ in terms of $a_1$ and $\lambda_{g,\ell}$. \end{rek}

Theorem \ref{thm:kangdistribution} holds in greater generality (see \cite{BILMT} for a generalization). In \S\ref{sec:kangaroo} we prove this case of the general result. As we are working with a very special recurrence, more is true, and we derive asymptotic values for $\glgl$ and hence the probabilities for $\ell$ and $g$ large. Our technique is combinatorial, using inclusion-exclusion and the indicator random variables $X_{i,i+j}(m)$. As an additional example of the power of these methods, in \S\ref{sec:fardiff} we apply these techniques to the far-difference representation of Alpert \cite{Al} (she proved every integer can be written uniquely as a sum of signed Fibonacci summands, subject to certain constraints), determining the distribution of gaps in Theorem \ref{thm:fardiffgap}. We conclude with some remarks on open and related problems.

\section{Distribution of Gaps in Kangaroo Recurrences}\label{sec:kangaroo}

\subsection{Preliminaries}




Before proving Theorem \ref{thm:kangdistribution} we isolate an important observation, which greatly simplifies the calculations.

\begin{lem}\label{lem:kangshortgap} Let $\{K_n\}$ be a Kangaroo recurrence with $\ell$ hops of length $g$. Then $P_n(j) = 0$ for all $j < g$; in other words, there cannot be a gap of length less than $g$. \end{lem}

\begin{proof} This follows immediately from the definition of a legal decomposition; specifically, from \eqref{eq:legalcondition2}, which prevents us from choosing any of the next $g-1$ summands after a chosen summand $K_p$. \end{proof}


Our combinatorial attack requires us to compute $X_{i,i+j}(n)$ to find $P_n(j)$. We can find $\Xn$ by counting the number of choices of the summands $\{K_1, K_2, \dots, K_n\}$ such that $K_i, K_{i+j}$ and $K_n$ are chosen, no summand whose index is between $i$ and $i + j$ is chosen, and all other indices are free to be chosen subject to the requirement that we have a legal decomposition. Let $\Ln$ and $\Rn$ be the number of ways to choose a valid subset of summands from those before the gap of length $j$ starting at $K_i$ and after the gap (respectively). We note $\Ln$ and $\Rn$ are independent of each other when $j>g$; thus \emph{for $j > g$ we have} \be \Xn\ = \ \Ln \cdot \Rn. \ee As \be K_{n+1} \ =  \ K_n + K_{n-g} + \cdots + K_{n-\ell g}, \ee any time we have a gap of length $j > g$ the recurrence `resets' itself. Thus we have the following lemma.

\begin{lem}\label{lem:countlem1} Let $\{K_n\}$ be a Kangaroo recurrence with $\ell$ hops of length $g$. Consider all $m \in \interval$ with a gap of length $j > g$ starting at $K_i$. The number of valid choices for subsets of summands before the gap, $\Ln$, is given by
\be
\Ln \ = \
K_{i+1} - K_{i},
\ee
and the number of valid choices for subsets of summands after the gap, $\Rn$, is given by
\be
\Rn \ = \
K_{n-i-j+2} - 2K_{n - i - j + 1} +K_{n-i-j}.
\ee
\end{lem}

\begin{proof}
To count $\Ln$, we count the number of ways to have a legal decomposition that must have the summand $K_i$, and where all other summand choices beforehand are free. It is very important that $j > g$, as this means the summand at $K_{i+j}$ does not interact with the summands earlier than $K_i$ through the recurrence relation. Thus $\Ln$ is the same as the number of legal choices of summands from $\{K_1, K_2, \dots, K_i\}$ with $K_i$ chosen. As each integer in $[K_i, K_{i+1})$ has a unique legal decomposition, we see $\Ln$ equals the number of elements in this interval, which is just $K_{i+1} - K_i$.

To compute $\Rn$, we need to consider how many ways we can choose summands from $\{K_{i+j}, K_{i+j+1}, \dots, K_n\}$ such that $K_i$ and $K_n$ are chosen and the resulting decomposition is legal; since $j > g$ the summands from $K_i$ and earlier cannot affect our choices here. Thus our problem is equivalent to asking how many legal ways there are to choose summands from $\{K_1, K_2, \dots, K_{n-i-j+1}\}$ with $K_1, K_{n-i-j+1}$ both chosen and the rest free. There are many ways to compute this; the simplest is to note that this equals the number of legal choices where we \emph{may or may not choose $K_1$}, minus the number of legal choices where we \emph{do not choose $K_1$}. By a similar argument as above, the first count is $K_{n-i-j+2}-K_{n-i-j+1}$ (as it is the number of legal representations of a number in $[K_{n-i-j+1}, K_{n-i-j+2})$), while the second is $K_{n-i-j+1} - K_{n-i-j}$. The proof is completed by subtracting. \end{proof}

Note that in the above lemma we only compute $\Ln$ and $\Rn$ when the gap $j > g$. The reason is that if we can determine $P_n(j)$ for $j > g$, then since $P_n(j) = 0$ for $j<g$ the law of alternatives implies $P_n(g) = 1 - \sum_{j > g} P_n(j)$. Thus if we can show each limit $P(j)$ exists when $j > g$, then the limit $P(g)$ exists as well. Hence, in some sense the following lemma is not needed. We prefer to give it as it allows us to determine the ratio $a_1/C_{{\rm Lek}}$, which we would not be able to do otherwise.

\begin{lem}\label{lem:countlem2} Let $\{K_n\}$ be a Kangaroo recurrence with $\ell$ hops of length $g$. Consider all $m \in \interval$ with \emph{exactly} $b$ consecutive gaps of length $g$ starting at $K_i$. Then $b \le \ell-1$ and number of valid choices for subsets of summands before the gap, $L_{i,g,b}(n)$, equals
\be
L_{i,g,b}(n) \ = \
K_{i-g},
\ee
and the number of valid choices for subsets of summands after the gap, $R_{i,g,b}(n)$, equals
\be
R_{i,g,b}(n) \ = \
K_{n-i-(b+1)g+1} - K_{n-i-(b+1)g}.
\ee
\end{lem}

\begin{proof} We must have $b \le \ell -1$, as if $b \ge \ell$ then we could use the recurrence relation. Remember by Lemma \ref{lem:kangshortgap} the smallest gap is at least $g$. As we have exactly $b$ consecutive gaps of length $g$ starting at $K_i$, the first summand chosen to the \emph{left} of $K_i$ is at most $K_{i-g-1}$, while the first summand chosen to the \emph{right} of $K_{i+bg}$ is at least $K_{i+(b+1)g+1}$. This implies that $L_{i,g,b}(n)$ equals the number of ways to legally choose summands from $\{K_1, K_2, \dots, K_{i-g-1}\}$; however, unlike our previous case now we do not need to choose $K_{i-g-1}$. We claim the number of valid choices is $K_{i-g}$. To see this, note that we can condition on the largest element being $K_r$ for $1 \le r \le i-g-1$. Each $r$ corresponds to the number of legal decompositions in $[K_r, K_{r+1})$, yielding $K_{r+1}-K_r$ choices. We sum over $r$, and the sum telescopes to $K_{i-g}-K_1 = K_{i-g}-1$. We then add 1 to take into account the case where no summands are chosen, proving the claim.

Similarly, by shifting indices we find $R_{i,g,b}(n)$ is equivalent to the number of ways to legally choose summands from $\{K_{i+(b+1)g+1}, \dots, K_n\}$ with $K_n$ chosen. Equivalently, this equals the number of legal choices of summands from $\{K_1, K_2, \dots, K_{n-i-(b+1)g}\}$ with $K_{n-i-(b+1)g}$ chosen, which is just $K_{n-i-(b+1)g+1} - K_{n-i-(b+1)g}$. \end{proof}

\subsection{Proof of Theorem \ref{thm:kangdistribution}}

We now prove our main result. We use little-oh and big-Oh notation for the lower order terms, which do not matter in the limit. If \be \lim_{x \to \infty} \frac{F(x)}{G(x)} \ = \ 0, \ee we write $F(x) = o(G(x))$ and say $F$ is little-oh of $G$, while if there exist $M,x_0 > 0$ such that
$|F(x)| \leq M G(x)$ for all $x > x_0$ we write $F(x) = O(G(x))$ and say $F$ is big-oh of $G$. In particular, $o(1)$ represents a term that decays to zero as $n\to\infty$, while $O(1)$ represents a term bounded by a constant.

\begin{proof}[Proof of Theorem \ref{thm:kangdistribution}] There are three cases: $j < g$, $j=g$ and $j > g$. The first case is the easiest, as $P_n(j) = 0$ for all $j < g$ by Lemma \ref{lem:kangshortgap}. \\

We now consider $j>g$. We need to compute $\lim_{n \to \infty} \frac{1}{Y(n)}\sum_{i=1}^{n-j}X_{i, i+j}(n)$.  By Lemma \ref{lem:countlem1}, \be \Xn \ = \  \Ln \cdot \Rn \ = \ \left(K_{i+1}-K_{i}\right) \cdot \left(K_{n-i-j+2} - 2K_{n - i - j + 1} +K_{n-i-j}\right). \ee By Lemma \ref{lem:genbinet}, \be K_r\ =\ a_1 \glgl^r + O\left(r^{L-2} \lambda_2^r\right),\ee with $|\lambda_2| < \glgl$. Thus \bea \Xn & \ = \ & a_1 \glgl^i (\glgl - 1) \left(1 + o(1)\right) \cdot a_1 \glgl^{n-i-j} (\glgl^2 - 2\glgl + 1) \left(1 + o(1)\right) \nonumber\\ &=& a_1^2 (\glgl - 1)^3 \glgl^n \cdot \glgl^{-j} \left(1 + o(1)\right). \eea As \be Y(n)\ =\ \left(C_{{\rm Lek}} n + O(1)\right) (K_{n+1}-K_n) \ = \ a_1 C_{{\rm Lek}} n (\glgl - 1) \glgl^n \left(1 + o(1)\right) \ee and for any fixed $j$ the sum over $i$ is $n + O(1)$, we find \bea P_n(j) & \ = \ & \frac{\sum_{i=1}^{n-j} X_{i,i+j}(n)}{Y(n)} \nonumber\\ &=& \frac{\sum_{i=1}^{n-j} a_1^2 (\glgl - 1)^3 \glgl^n \cdot \glgl^{-j} \left(1 + o(1)\right)}{a_1 C_{{\rm Lek}} n (\glgl - 1) \glgl^n \left(1 + o(1)\right)} \ = \ \frac{a_1 (\glgl - 1)^2}{C_{{\rm Lek}}} \glgl^{-j} \left(1 + o(1)\right),\ \ \ \eea and the limit clearly exists for each $n$ and each $j > g$.\\

We now turn to $j=g$. We avoid double counting by using the $L_{i,g,b}(n)$ and $R_{i,g,b}(n)$ from Lemma \ref{lem:countlem2}. To find the number of gaps of length $g$, we look at all strings of $b \le \ell-1$ consecutive gaps of length $g$ over all possible starting places $i$, and find \bea P_n(g) & \ = \ & \frac{\sum_{i=1}^{n-g} \sum_{b=1}^{\ell-1} L_{i,g,b}(n) R_{i,g,b}(n)}{Y(n)}. \eea We now substitute for $L_{i,g,b}(n)$ and $R_{i,g,b}(n)$ with the values from Lemma \ref{lem:countlem2}. The proof is completed by using the generalized Binet formula (Lemma \ref{lem:genbinet}), the geometric series formula, and collecting the terms.
\end{proof}

\begin{rek} As a safety check, we confirm the theorem for the Fibonacci and Tribonacci numbers (the formulas for the roots and $a_1$ are at the end of Appendix \ref{sec:proofgenbinet}, and the value of $C_{{\rm Lek}}$ follows from \cite{MW1}). For the Fibonacci numbers, $\lambda_1 = \phi = (1+\sqrt{5})/2$, $a_1 = \phi/\sqrt{5}$ (remember these are the `shifted' Fibonacci numbers as originated with Zeckendorf and Lekkerkerker), and $C_{{\rm Lek}} = 1/(\phi^2+1)$. We see $P(1) = 0$ and the sum of the other probabilities is 1. For the Tribonacci numbers, $P(1)$ is no longer zero. We have $C_{{\rm Lek}} = a_1 (3 \lambda_1^2 - 1) / (\lambda_1^3 (\lambda_1^2 - 1))$, which does lead to the probabilities summing to 1. \end{rek}

\subsection{Approximating $\glgl$ and Relevant Probability Ratios}\label{sec:approxglgl}

We end this section with some quick approximations for $\glgl$ when $\ell \to \infty$. The ability to isolate such results is one of the motivations for studying these special recurrences.

\begin{lem}\label{lem:approxglgl1}
For a Kangaroo recurrence of $\ell$ hops of length $g$, for large $\ell, g$ we have $\glgl$ $\approx$ $\left(1+\frac{\alpha}{g}\right)$,
where $\alpha \approx \log(g)-\log(\log(g))+\frac{\log(\log(g))}{\log(g)}$. In particular, for large $\ell$ and $g$ we have \be \glgl \ \approx \ 1, \ \ \ \glgl^{-g} \ \approx \ e^{-\alpha(g)} \ \approx \ \frac{\log g}{g}. \ee
\end{lem}

\begin{proof} We drop all lower order terms in the arguments below, as our goal is to highlight the limiting behavior. We derive a transcendental equation from the characteristic polynomial of the Kangaroo recurrence. Notice that $K_{n+1} \approx \glgl K_n$. Consider the Kangaroo recurrence $K_{n+1}=K_n + K_{n-g}+ \cdots + K_{n-\ell g}$. The left hand side approximately equals $\glgl K_n$, and for $\ell$ large the right hand side is essentially
\begin{equation}
K_n \sum\limits_{m=0}^\ell \glgl^ {-mg}\ \approx\ K_n \sum\limits_{m=0}^\infty \glgl^{-mg} \ = \  \frac{K_n}{1 - \glgl^{-g}}.
\end{equation}
Therefore
\bea
 \glgl &\ \approx \ & \frac{1}{1- \glgl^{-g}}, \eea which implies \be\label{eq:glglapproxsubt}  \glgl^g - \glgl^{g-1} -1 \ \approx \ 0.
\ee
From the generalized Binet formula we know that $\glgl$ grows exponentially with $g$. For large $g$ we can write \be \glgl^g \  = \  e^{\alpha(g)}, \ \ \ {\rm or} \ \ \ \glgl \ = \ e^{\alpha(g)/g} \ \approx \ 1 + \frac{\alpha(g)}{g}. \ee Substituting into \eqref{eq:glglapproxsubt} yields
\be
e^{\alpha(g)} -e^{\alpha(g)} \left(1+\frac{\alpha(g)}{g}\right)^{-1} -1\ \approx\ 0,
\ee
which we rewrite as
\bea
\alpha(g) e^{\alpha(g)} \ \approx\  g.
\eea
By direct substitution we see that for $g$ sufficiently large, \be \alpha(g)\ \approx\ \log(g)-\log(\log(g))+\frac{\log(\log(g))}{\log(g)}.\ee
\end{proof}

Now that we have the probability distribution for gaps for Kangaroo recurrences, we can study the ratios of the probabilities of certain events. Since gaps of length $g$ are a threshold event, a natural question to ask is what is the probability of getting a gap of length $g$ compared to any other gap length?

\begin{lem}
For large $g$ and $\ell, n \to\infty$, the ratio of the probability of obtaining a gap of length $g$ to a gap of length exceeding $g$ is
\begin{equation}
\frac{{\rm Prob(gap\ of \ length\ }g)}{{\rm Prob(gap\ of\ length\ at\ least\ }g+1)} \ \approx \ \frac{\log g}{g}.
\end{equation}
\end{lem}

\begin{proof}
From Theorem \ref{thm:kangdistribution}, we have
\begin{equation}
\frac{{\rm Prob(gap\ of \ }g)}{{\rm Prob(gap \ at \ least \ }g+1)} \ = \  \frac{\glgl^{-2g} \glgl^{-g} / (1 - \glgl^{-g})}{\glgl^{-g}(\glgl-1)}.
\end{equation}
For large $g$, $\ell$, and $n$, we use Lemma \ref{lem:approxglgl1} and some algebraic manipulation to deduce the claim.
\end{proof}

\section{Far-Difference Representations}\label{sec:fardiff}

We now consider a natural generalization, signed Zeckendorf decompositions. Alpert \cite{Al} proved that every positive integer has a unique representation as a sum of signed shifted Fibonacci numbers, where the gap between opposite signed summands must be at least 3, and between same signed summands must be at least 4; this is called the \textbf{far-difference} representation. Miller and Wang \cite{MW1} proved the generalized Lekkerkerker theorem holds here as well, and proved Gaussian behavior for the number of summands (it is a bivariate Gaussian, as there are positive and negative summands).

Our techniques generalize immediately and yield formulas for the distribution of average gaps. We concentrate on gaps between any adjacent summands, though similar reasoning would yield results restricted to gaps between same signed or opposite signed summands. Before stating and proving these generalizations, we first introduce some definitions and useful results.

Given an integer $m$, we write its far-difference representation by \be m \ = \ \epsilon_{i_k} F_{i_k} + \epsilon_{i_{k-1}} F_{i_{k-1}} + \cdots + \epsilon_{i_1} F_{i_1}, \ \ \ \epsilon_j \in \{-1,1\}. \ee Let $\mathcal{N}({\epsilon_i}F_i,{\epsilon_j}F_j)$ denote the number of numbers whose far-difference decomposition starts with ${\epsilon_i}F_i$ and ends with ${\epsilon_j}F_j$, where the $\epsilon$'s are $\pm 1$; similarly, let $\mathcal{N}({\epsilon_j}F_j)$ be the number of numbers whose decomposition ends with ${\epsilon_j}F_j$. We consider integers in the interval $(S_{n-1}, S_n]$, where \be S_n \ = \ F_n + F_{n-4} + F_{n-8} + F_{n-12} + \cdots; \ee the interval is a bit different than before as we have a signed decomposition, and have the ability to overshoot and then correct through subtraction.

\begin{lem}\label{lem:prereqforfardiff} We have
\be \mathcal{N}({\epsilon_i}F_i,{\epsilon_j}F_j)\ =\ \mathcal{N}({\epsilon_i}F_1,{\epsilon_j}F_{j-i+1})\ee
and
\be \mathcal{N}(-F_1,+F_j) + \mathcal{N}(+F_1,+F_j)\ =\ \mathcal{N}(+F_j)-\mathcal{N}(+F_{j-1}).\ee Note \be \mathcal{N}(+F_r) \ = \ S_r - S_{r-1},\ee and by symmetry $\mathcal{N}(+F_r) = \mathcal{N}(-F_r)$.
\end{lem}

\begin{proof} The proof is similar to the proofs from \S\ref{sec:kangaroo}, proceeding by shifting indices down for the first part, and inclusion-exclusion for the second. The final claim on $\mathcal{N}(+F_r)$ follows by noting this is the cardinality of $(S_{r-1}, S_r]$. \end{proof}

We can similarly determine the average gap behavior. For the Fibonacci numbers, Binet's formula (Lemma \ref{lem:genbinet}) is \be F_n \ = \ a_n \lambda_1^n + a_2 \lambda_2^n \ = \ \frac{\phi}{\sqrt{5}} \phi^n + \frac1{\phi\sqrt{5}} (-1/\phi)^n, \ee where $\phi = \frac{1+\sqrt{5}}2$ is the Golden mean.

\begin{thm}\label{thm:fardiffgap} As $n\to\infty$, the probability $P(j)$ of a gap of length $j$ in a far-difference decomposition of integers in $(S_{n-1}, S_n]$ converges to geometric decay for $j \ge 4$, with decay constant equal to the golden mean. Specifically, if $a_1 = \phi/\sqrt{5}$ (the Binet constant of the largest root, $\phi = (1+\sqrt{5})/2$), then $P(j) = 0$ if $j \le 2$ and  \be\twocase{P(j) \ = \ }{\frac{10a_1\phi}{\phi^4-1} \phi^{-k}}{if $j \ge 4$}{\frac{5a_1}{\phi^2(\phi^4-1)}}{if $j = 3$.} \ee \end{thm}

\begin{proof} We first count gaps of length 3 in far-difference representations of integers in $(S_{n-1}, S_n]$. Let $X_{i,i+3}(n)$ be the number of representations with a gap between of length 3 from $F_i$ to $F_{i+3}$ (we may assume $i+3 < n-4$ and not worry about boundary effects, which are lower order). We note that since the gap length is 3, the sign of the $F_i$ term is the opposite of the sign of the $F_{i+3}$ term.

This gives us two cases. In the first case, we have $+F_i$ and $-F_{i+3}$  in the decomposition, while in the second case we have $-F_i$ and $+F_{i+3}$. Thus there are $\mathcal{N}(+F_i) \mathcal{N}(-F_1,+F_{n-(i+2)})$ decompositions in case 1 and $\mathcal{N}(+F_i) \mathcal{N}(+F_1,+F_{n-(i+2)})$ decompositions in case 2.
By Lemma \ref{lem:prereqforfardiff} the sum of the number of legal decompositions in the two cases is
\begin{equation}
X_{i,i+3}(n) \ =\ \mathcal{N}(+F_i)\left(\mathcal{N}(+F_{n-(i+2)})-\mathcal{N}(+F_{n-(i+3)})\right).
\end{equation}
The proof for gaps of length 3 is completed by using $\mathcal{N}(+F_j)\ = \ S_j-S_{j-1}$, Binet's formula and the geometric series formula to show \be S_k \ = \  \frac{a_1\phi^{k}}{1\ -\ \frac{1}{\phi^4}} + O(1), \ee and the fact that the number of summands $Y(n)$ satisfies \be Y(n) \ = \ \left(\frac{1}{5}n\ +\ \frac{366-118\sqrt{5}}{20} + o(1)\right) (S_n - S_{n-1}) \ee (see \cite{MW1} for a proof of this last result). We simply substitute these results into \be P_n(j) \ = \  \frac{1}{Y(n)} \sum_{i=1}^{n-j}X_{i,i+j}(n), \ee simplify and then take the limit as $n\to\infty$.

The case of $j \ge 4$ is easier, as now the ends' parities are independent of each other. The claim follows from a similar calculation.
\end{proof}

\section{Conclusion and Future Work}

In this work we studied the average gap distribution arising from special positive linear recurrence sequences. The techniques generalize immediately to any such sequence where each recursion coefficient $c_i$ is positive (though it is harder to get good asymptotic expansions as in \S\ref{sec:approxglgl}); more involved formulas exist if this condition is not satisfied. It is straightforward to see there is geometric decay for gaps larger than the recurrence length; this is essentially due to the fact that the large separation makes the left and right parts independent. Using moment techniques, in the sequel paper \cite{BILMT} these claims are proved and extended further. They associate a gap measure to each integer in the interval $[G_n, G_{n+1})$ and show that as $n\to\infty$ almost surely each individual gap measure converges to the average gap measure, and determine the distribution of the longest gap between summands. It is quite interesting that the gap problems are significantly easier than counting the number of summands; this is very different than similar problems in random matrix theory, where the eigenvalue densities of many structured ensembles are known, but not the gaps between adjacent eigenvalues (or, if known, these results are very recent and require significantly more machinery than is needed for the densities).

The combinatorial vantage here, which is an outgrowth of \cite{KKMW,MW1}, is useful for a variety of other problems. Conditioning on the number of summands is useful in investigations of the longest gap in generalized Zeckendorf decompositions \cite{BILMT} and the asymptotic average of the number of terms in the Ostrowski $\alpha$-decomposition \cite{BCCSW}; one natural future project would be to study the distribution of gaps in the $\alpha$-decomposition, as well as considering this and more general signed decompositions. Additionally, the far-difference decomposition from \cite{Al} can be generalized to Kangaroo recurrences. This is being investigated in \cite{DM}, where similar results to the ones in this paper are proved.

\appendix

\section{Proof of Binet's Formula for Kangaroo Recurrences}\label{sec:proofgenbinet}

While the Generalized Binet's Formula, Lemma \ref{lem:genbinet}, follows from the Perron-Frobenius Theorem for irreducible matrices and some additional algebra, for completeness we provide an elementary proof. As it is no harder to prove the result in greater generality, we do so as this is needed in the sequel paper \cite{BILMT}; Lemma \ref{lem:genbinet} follows immediately by choosing appropriate values of the constants $c_i$.

\begin{thm}[Generalized Binet's Formula]\label{thm:genbinetf} Consider the linear recurrence \be G_{n+1} \ = \ c_1 G_n + c_2 G_{n-1} + \cdots + c_{L} G_{n+1-L}  \ee with the $c_i$'s non-negative integers and $c_1, c_{L} > 0$. Let $\lambda_1, \dots, \lambda_L$ be the roots of the characteristic polynomial  \be f(x) \ := \ x^{L} - \left(c_1 x^{L-1} + c_2 x^{L-2} + \cdots + c_{L-1} x + c_L\right) \ = \ 0, \ee ordered so that $|\lambda_1| \ge |\lambda_2| \ge \cdots \ge |\lambda_L|$. Then $\lambda_1 > |\lambda_2| \ge \cdots \ge |\lambda_L|$, $\lambda_1 > 1$ is the unique positive root, and there exist constants such that \be G_n \ = \ a_1 \lambda_1^n + O\left(n^{L-2} \lambda_2^n\right). \ee More precisely, if $\lambda_1, \omega_2, \dots, \omega_r$ denote the distinct roots of the characteristic polynomial with multiplicities 1, $m_2, \dots, m_r$, then there are constants $a_1 > 0, a_{i,j}$ such that \be G_n \ = \ a_1 \lambda_1^n + \sum_{i=2}^r \sum_{j=1}^{m_r} a_{i,j} n^{j-1} \omega_i^n. \ee \end{thm}

\begin{proof} We break the proof into three steps. We first analyze the positive roots (and show there is only one), then show the remaining roots are smaller in absolute value, and conclude by deducing the expansion.\\

\noindent \emph{Step 1: There is a unique positive root, which is simple.} We show there is a unique positive simple root, and further that it is larger than 1. Note $f(x) < 0$ for $0 \le x \le 1$, and \be xf'(x) \ = \ L x^{L} - \left((L-1) c_1 x^{L-1} + (L-2) c_2 x^{L-2} + \cdots + c_{L-1} x \right).\ee If $x$ is sufficiently large then $f(x) > 0$, and thus by the Intermediate Value Theorem there is a $\lambda_1 > 1$ such that $f(\lambda_1) = 0$. We show that $\lambda_1$ is simple and further that it is the only positive root. We claim $f(\lambda_1) = 0$ implies $\lambda_1 f'(\lambda_1) > 0$. This is because when we differentiate the $L$ coming down from $x^{L}$ exceeds the powers from the remaining terms (which are all of the same sign as the coefficients are positive and $\lambda_1 > 1$), and thus \be \lambda_1 f'(\lambda_1) \ > \ L f(\lambda_1) \ = \ 0.\ee This immediately yields $\lambda_1$ is a simple root (as otherwise $f'(\lambda_1) = 0$). A similar calculation shows $f'(x) > 0$ for $x \ge \lambda_1$ (as the $x^{L}$ grows faster than all other powers of $x$). Thus $f$ is increasing for $x \ge \lambda_1$. As $f(\lambda_1) = 0$ we find $f(x) > 0$ for $x>\lambda_1$, and there is a unique positive root.\\


\noindent \emph{Step 2: All the other roots are less than $\lambda_1$ in absolute value.} Let $f(\lambda) = 0$ with $\lambda \neq \lambda_1$; by Step 1 this implies $\lambda$ is either negative or it has a non-zero imaginary component. Assume $|\lambda| \ge \lambda_1$. From $f(\lambda) = 0$ we find \bea |\lambda|^{L} & \ = \ & \left|c_1\lambda^{L-1} + \cdots + c_{L-1}\lambda + c_L\right| \nonumber\\ & \le & c_1|\lambda|^{L-1} + \cdots + c_{L-1}|\lambda| + c_L \nonumber\\ & = & |\lambda|^{L} - f(|\lambda|) \nonumber\\ & \le & |\lambda|^{L}, \eea where the first inequality follows from the triangle inequality (and is an equality if and only if each quantity has the same phase) and the second inequality is due to the fact that $f$ is non-negative for $x \ge \lambda_1$ (it is an equality if and only if $|\lambda|=\lambda_1$). We therefore obtain a contradiction ($|\lambda|^{L} < |\lambda|^{L}$) if $|\lambda| > \lambda_1$.

We are left with showing that we cannot have $|\lambda| = \lambda_1$. We obtain the same contradiction unless each term in the sum has the same phase, which must be the same phase as $\lambda^L$: \be \lambda^L \ = \ c_1 \lambda^{L-1} + \cdots + c_L.\ee As $c_L$ is positive, this requires each $\lambda^r$ to be positive. As $c_1$ is positive, $\lambda^L$ and $c_1\lambda^{L-1}$ must both be positive; this forces $\lambda > 0$, which contradicts $\lambda$ being either negative or  having a non-zero imaginary component. Thus $|\lambda| < \lambda_1$ if $\lambda \neq \lambda_1$. \\

\noindent \emph{Step 3: Generalized Binet expansion.} We proved that there is a unique simple positive root of the characteristic polynomial, and all other roots are strictly less in absolute value. The claimed expansion follows from standard results on solving linear recurrence relations (see for example Section 3.7 of \cite{Go} for proofs). Briefly, linear combinations of solutions are solutions. If a root $\lambda$ is simple it generates the corresponding solution $\lambda^n$, while if it has multiplicity $r$ it generates solutions $\lambda^n, n \lambda^n, \dots, n^{r-1}\lambda^n$. The general solution is a linear combination of these, with coefficients chosen to match the initial conditions. \end{proof}

\ \\

We give the expansions from Theorem \ref{thm:genbinetf} for three important Kangaroo sequences: Fibonacci, Tribonacci, and what we call the Skiponacci sequences.\\

\noindent \emph{Fibonaccis:} For the recurrence $F_{n + 1} = F_n + F_{n - 1}$ with $F_1 = 1, F_2 = 2$ we have $F_n = a_1\lambda_1^n + a_2\lambda_2^n$, where
\be
\lambda_1 \; = \; \frac{1 + \sqrt{5}}{2}, \quad \lambda_2 \; = \;  \frac{1 - \sqrt{5}}{2}, \quad a_1 \; = \; \frac{5 + \sqrt{5}}{10}, \quad \text{and} \quad a_2 \; = \; \frac{5 - \sqrt{5}}{10}.
\ee

\ \\

\noindent \emph{Tribonaccis:} For the recurrence $T_{n + 1} = T_n + T_{n - 1} + T_{n - 2}$ with $T_1 = 1, T_2 = 2, T_3 = 4$ we have $T_n = a_1\lambda_1^n + a_2\lambda_2^n + a_3\lambda_3^n$, where
\begin{align}
\lambda_1 & \; = \; \frac{1}{3}\left(1 + \left(19 + 3\sqrt{33}\right)^{1/3} + \left(19 - 3\sqrt{33}\right)^{1/3}\right)\nonumber\\
\lambda_2 & \; = \; \frac{1}{3} - \frac{1}{6}\left(1 + i\sqrt{3}\right)\left(19 - 3\sqrt{33}\right)^{1/3} - \frac{1}{6}\left(1 - i\sqrt{3}\right)\left(19 + 3\sqrt{33}\right)^{1/3} \nonumber\\
\lambda_3 & \; = \; \frac{1}{3} - \frac{1}{6}\left(1 - i\sqrt{3}\right)\left(19 - 3\sqrt{33}\right)^{1/3} - \frac{1}{6}\left(1 + i\sqrt{3}\right)\left(19 + 3\sqrt{33}\right)^{1/3},
\end{align}
and \tiny
\bea a_1 & \; = \; & \frac{1}{162\sqrt{11}}\Bigg(54\sqrt{11} + \left(59\sqrt{3}\right)\left(19 + 3\sqrt{33}\right)^{1/3} + \left(-32\sqrt{3} + 18\sqrt{11}\right)\left(19 + 3\sqrt{33}\right)^{2/3}\nonumber\\ & & \ \ \ + \left(19 - 3\sqrt{33}\right)^{2/3}\left(32\sqrt{3} + 18\sqrt{11} + \left(19\sqrt{3}\right) \left(19 + 3\sqrt{33}\right)^{1/3}\right) - \left(\sqrt{3}\right)\left(19 - 3\sqrt{33}\right)^{1/3}\left(59 + 19\left(19 + 3\sqrt{33}\right)^{2/3}\right)\Bigg).\nonumber\\ 
\eea \normalsize

\ \\

\noindent \emph{Skiponaccis:} For the recurrence $S_{n + 1} = S_n + S_{n - 2}$ with $S_1 = 1, S_2 = 2, S_3 = 3$ we have $S_n = a_1\lambda_1^n + a_2\lambda_2^n + a_3\lambda_3^n$, where
\small
\begin{align}
\lambda_1 & \; = \; \frac{1}{3}\left(1 + \left(\frac{1}{2}\left(29 - 3\sqrt{93}\right)\right)^{1/3} + \left(\frac{1}{2}\left(29 + 3\sqrt{93}\right)\right)^{1/3}\right) \nonumber\\
\lambda_2 & \; = \; \frac{1}{3} - \frac{1}{6}\left(1 + i\sqrt{3}\right)\left(\frac{1}{2}\left(29 - 3\sqrt{93}\right)\right)^{1/3} + \frac{1}{6}\left(-1 + i\sqrt{3}\right)\left(\frac{1}{2}\left(29 + 3\sqrt{93}\right)\right)^{1/3} \nonumber\\
\lambda_3 & \; = \; \frac{1}{3} + \frac{1}{6}\left(-1 + i\sqrt{3}\right)\left(\frac{1}{2}\left(29 - 3\sqrt{93}\right)\right)^{1/3} - \frac{1}{6}\left(1 + i\sqrt{3}\right)\left(\frac{1}{2}\left(29 + 3\sqrt{93}\right)\right)^{1/3},
\end{align} \normalsize
and \tiny
\bea
a_1 & \; = \; & \frac{1}{162\sqrt{31}}\Bigg(54\sqrt{31} + 44\cdot2^{2/3}\sqrt{3}\left(29 + 3\sqrt{93}\right)^{1/3} + 2^{1/3}\left(-23\sqrt{3} + 9\sqrt{31}\right)\left(29 + 3\sqrt{93}\right)^{2/3} + \left(29 - 3\sqrt{93}\right)^{2/3}\nonumber\\ & & \ \ \  \cdot \left(2^{1/3}\left(23\sqrt{3} + 9\sqrt{31}\right) + 13\sqrt{3}\left(29 + 3\sqrt{93}\right)^{1/3}\right) - \sqrt{3}\left(29 - 3\sqrt{93}\right)^{1/3}\left(44\cdot2^{2/3} + 13\left(29 + 3\sqrt{93}\right)^{2/3}\right)\Bigg).\nonumber\\
\eea\normalsize


\medskip

\noindent MSC2010: 11B39, 11B05  (primary) 65Q30, 60B10 (secondary)

\ \\


\begin{thebibliography}{99}

\bibitem{Al}
H. Alpert,  \emph{Differences of multiple Fibonacci numbers}, Integers: Electronic Journal of Combinatorial Number Theory  \textbf{9} (2009), 745--749.

\bibitem{BILMT}
A. Bower, R. Insoft, S. Li, S. J. Miller and P. Tosteson, \emph{The Distribution of Gaps between Summands in Generalized Zeckendorf Decompositions}, preprint.

\bibitem{BCCSW}
E. Burger, D. C. Clyde, C. H. Colbert, G. H. Shin and Z. Wang, \emph{A Generalization of a Theorem of Lekkerkerker to Ostrowski's Decomposition of Natural Numbers}, Acta Arith. \textbf{153} (2012), 217--249.



\bibitem{DM}
P. Demontigny and S. J. Miller, \emph{Distribution of summands in far-difference decompositions}, preprint.

\bibitem{DG}
M. Drmota and J. Gajdosik, \emph{The distribution of the sum-of-digits function},
J. Th\'eor. Nombr\'es Bordeaux \textbf{10} (1998), no. 1, 17--32.

\bibitem{EK}
P. Erd\H{o}s and M. Kac,  \emph{The Gaussian Law of Errors in the Theory of Additive Number Theoretic Functions}, American Journal of Mathematics \textbf{62} (1940), no. 1/4, pages 738--742.

\bibitem{FGNPT}
P. Filipponi, P. J. Grabner, I. Nemes, A. Peth\"o, and R. F. Tichy, \emph{Corrigendum
to: ``Generalized Zeckendorf expansions''}, Appl. Math. Lett.,
\textbf{7} (1994), no. 6, 25--26.


\bibitem{Go}
S. Goldberg, \emph{Introduction to Difference Equations}, John Wiley \& Sons, 1961.

\bibitem{GT}
P. J. Grabner and R. F. Tichy, \emph{Contributions to digit expansions with respect to linear
recurrences}, J. Number Theory \textbf{36} (1990), no. 2, 160--169.

\bibitem{GTNP}
P. J. Grabner, R. F. Tichy, I. Nemes, and A. Peth\"o, \emph{Generalized Zeckendorf expansions},
Appl. Math. Lett. \textbf{7} (1994), no. 2, 25--28.


\bibitem{Ho} V. E. Hoggatt,  \emph{Generalized Zeckendorf theorem}, Fibonacci Quarterly \textbf{10} (1972), no. 1 (special issue on representations), pages 89--93.

\bibitem{Ke} T. J. Keller,  \emph{Generalizations of Zeckendorf's theorem}, Fibonacci Quarterly \textbf{10} (1972), no. 1 (special issue on representations), pages 95--102.

\bibitem{LT}
M. Lamberger and J. M. Thuswaldner, \emph{Distribution properties of digital expansions
arising from linear recurrences}, Math. Slovaca \textbf{53} (2003), no. 1, 1--20.

\bibitem{Len} T. Lengyel, \emph{A Counting Based Proof of the Generalized Zeckendorf's Theorem}, Fibonacci Quarterly \textbf{44} (2006), no. 4, 324--325.

\bibitem{Lek} C. G. Lekkerkerker,  \emph{Voorstelling van natuurlyke getallen door een som van getallen van Fibonacci}, Simon Stevin \textbf{29} (1951-1952), 190--195.

\bibitem{KKMW} M. Kolo$\breve{{\rm g}}$lu, G. Kopp, S. J. Miller and Y. Wang,  \emph{On the number of summands in Zeckendorf decompositions}, Fibonacci Quarterly \textbf{49} (2011), no. 2, 116--130.



\bibitem{MW1} S. J. Miller and Y. Wang,  \emph{From Fibonacci Numbers to Central Limit Type Theorems}, Journal of Combinatorial Theory, Series A \textbf{119} (2012), no. 7, 1398--1413.

\bibitem{MW2} S. J. Miller and Y. Wang,  \emph{Gaussian Behavior in Generalized Zeckendorf Decompositions},  to appear in the conference proceedings of the 2011 Combinatorial and Additive Number Theory Conference. \texttt{http://arxiv.org/abs/1107.2718}.

\bibitem{Ste1}
W. Steiner, \emph{Parry expansions of polynomial sequences}, Integers \textbf{2} (2002), Paper A14.

\bibitem{Ste2}
W. Steiner, \emph{The Joint Distribution of Greedy and Lazy Fibonacci Expansions}, Fibonacci Quarterly \textbf{43} (2005), 60--69.


\bibitem{Ze} E. Zeckendorf, \emph{Repr\'esentation des nombres naturels par une somme des nombres de Fibonacci ou de nombres de Lucas}, Bulletin de la Soci\'et\'e Royale des Sciences de Li\`ege \textbf{41} (1972), pages 179--182.



\end{thebibliography}
\end{document}